\documentclass[12pt, leqno,twoside]{article}
\usepackage{amssymb}
\usepackage{amsmath}
\usepackage{amsthm}
\usepackage{graphicx}

\usepackage[active]{srcltx}

\allowdisplaybreaks

\def\dist{{\mathop\mathrm{\,dist\,}}}

\def\bint{{\ifinner\rlap{\bf\kern.25em--}
\int\else\rlap{\bf\kern.45em--}\int\fi}\ignorespaces}

\def\bbint{{\ifinner\rlap{\bf\kern.25em--}
\hspace{0.078cm}\int\else\rlap{\bf\kern.45em--}\int\fi}\ignorespaces}

\def\diam{{\mathop\mathrm{\,diam\,}}}

\newtheorem{thm}{Theorem}[section]
\newtheorem{lem}{Lemma}[section]
\newtheorem{prop}{Proposition}[section]

\newtheorem{cor}{Corollary}[section]
\newtheorem{defn}{Definition}[section]

\numberwithin{equation}{section}

\begin{document}

\arraycolsep=1pt

\title{\Large\bf Product of extension domains is still an extension domain
}
\author{Pekka Koskela and Zheng Zhu }
\date{ }
\maketitle
\begin{abstract}
Our main result Theorem \ref{theorem2} gives the following functional property of the class of $W^{1,p}$-extension domains. Let $\Omega_{1}\subset\mathbb{R}^{n}$ and $\Omega_{2}\subset\mathbb{R}^{m}$ both be $W^{1,p}$-extension domains for some $1<p\leq\infty$. We prove that $\Omega_{1}\times\Omega_{2}\subset\mathbb{R}^{n+m}$ is also a $W^{1,p}$-extension domain. We also establish the converse statement.
\end{abstract}
\section{Introduction}

Let $\Omega\subset\mathbb{R}^{n}$ be an open set. For $1< p\leq \infty$, we let $W^{1,p}(\Omega)$ denote the corresponding Sobolev space of all functions $u\in L^{p}(\Omega)$ whose first order distributional partial derivatives on $\Omega$ belong to $L^{p}(\Omega)$. This space is normed by 
\begin{equation}
\|u\|_{W^{1,p}(\Omega)}:=\sum_{0\leq|\alpha|\leq 1}\|D^{\alpha}u\|_{L^{p}(\Omega)}.\nonumber
\end{equation}
We say that $u\in L^{p}(\Omega)$ is $ACL$ ($absolutely\ continuous\ on\ lines$), if $u$ has a representative $\widetilde{u}$ that is absolutely continuous on almost all line segments in $\Omega$ parallel to the coordinate axes. Then $u\in W^{1,p}(\Omega)$ if and only if $u$ belongs to $L^{p}(\Omega)$ and has a representative $\widetilde{u}$ which is $ACL$ and whose (classical) partial derivatives belong to $L^{p}(\Omega)$, see e.g. Theorem A.15 in \cite{Pekka} and Theorem 2.1.4 in \cite{ziemer}.

We say that $\Omega\subset\mathbb{R}^{n}$ is a $W^{1,p}$-extension domain if there exists a constant $C\geq 1$ which only depends on $\Omega,n,p$ such that for every $u\in W^{1,p}(\Omega)$ there exists a function $Eu\in W^{1,p}(\mathbb{R}^{n})$ with $Eu\big|_{\Omega}\equiv u$ and 
\begin{equation}
\|Eu\|_{W^{1,p}(\mathbb{R}^{n})}\leq C\|u\|_{W^{1,p}(\Omega)}.\nonumber
\end{equation}
For example, every Lipschitz domain is a $W^{1,p}$-extension domain for all $1\leq p\leq\infty$ by the result of Calder\'on and Stein \cite{Stein}. It is easy to give examples of domains that fail to be extension domains, for example, the slit disk $\Omega:= D^{2}(0,1)\setminus\{(x_{1},0):0\leq x_{1}<1\}$. In general, the extension property for a fixed $\Omega$ may depend on the value of $p$, see \cite{Pekka98}, \cite{Shvar2} and \cite{KRZ}.

In \cite{HKT}, it was shown that any bi-Lipschitz image of a $W^{1,p}$-extension domain is also a $W^{1,p}$-extension domain: if $\Omega\subset\mathbb{R}^{n}$ is a $W^{1,p}$-extension domain and $f:\Omega\rightarrow\Omega'\subset\mathbb{R}^{n}$ is bi-Lipschitz, then $\Omega'$ is also a $W^{1,p}$-extension domain. Our main result gives another functional property of Sobolev extension domain.

 \begin{thm}\label{theorem2}
Let  $1<p\leq\infty$. If $\Omega_{1}\subset\mathbb{R}^{n}$ and $\Omega_{2}\subset\mathbb{R}^{m}$ are $W^{1,p}$-extension domains, then 
$\Omega_{1}\times\Omega_{2}\subset\mathbb{R}^{n+m}$ is also a 
$W^{1,p}$-extension domain. Conversely, if $\Omega_{1}\subset\mathbb{R}^{n}$ and $\Omega_{2}\subset\mathbb{R}^{m}$ are domains so that $\Omega_{1}\times\Omega_{2}\subset\mathbb{R}^{n+m}$ is a 
$W^{1,p}$-extension domain, then both $\Omega_1$ and $\Omega_2$ are necessarily
$W^{1,p}$-extension domains.
\end{thm}

According to Theorem 7 in \cite{HKT} (see \cite{zobin} for related results), 
a domain $\Omega$ is a 
$W^{1,\infty}$-extension domain if and only if it is uniformly locally 
quasiconvex, that is, there exist  positive constants $C$ and $R$, such that 
for all $x,y\in\Omega$ with $|x-y|<R$, there exists a curve 
$\gamma_{x,y}\subset\Omega$ from $x$ to $y$ with 
\begin{equation}
l(\gamma_{x,y})\leq C |x-y|.\nonumber
\end{equation}
Here $l(\gamma_{x,y})$ is the length of the curve $\gamma_{x,y}$. It is easy to 
check that the product of uniformly locally quasiconvex domains is still 
uniformly locally quasiconvex, and hence we only need to prove the first part of
Theorem \ref{theorem2} for $1<p<\infty$.

Our proof of the first part of Theorem \ref{theorem2} is based on the 
existence of an explicit 
extension operator constructed by Shvartsman in \cite{Shvar}. A result from 
\cite{HKT} allows us to employ  this operator. 
This procedure could in principle 
also be tried for the case of the higher order Sobolev spaces $W^{k,p}$, 
$k\geq 2$, but one does not seem to obtain suitable norm estimates. 
We would like to know whether the first part of Theorem \ref{theorem2} extends to the case of 
higher order Sobolev spaces or not; the second part does extend as can be seen from
our proof below.

\section{Preliminaries}

\subsection{Definitions and preliminary results}
Our notation is fairly standard. Throughout the paper $C,C_{1},C_{2},...$ or $\gamma,\gamma_{1},\gamma_{2},...$ will be generic positive constants which depend only on the dimension $n$, the domain $\Omega$ and indexes of spaces (s, p, q, etc.). These constants may change even in a single string of estimates. The dependence of a constant on certain parameters is expressed, for example, by the notation $\gamma=\gamma(n,p)$. We write $A\approx B$ if there is a constant $C\geq 1$ such that $A/C\leq B\leq CA$.

\begin{defn}\label{Ahlfors}
A measurable set $A\subset\mathbb{R}^{n}$ is said to be Ahlfors regular $($shortly, regular$)$ if there are constants $C_{A}\geq 1$ and $\delta_{A}>0$ such that, for every cube $Q$ with center in $A$ and with diameter $\diam Q\leq\delta_{A}$, we have
\begin{equation}\label{density}
|Q|\leq C_{A}|Q\cap A|.
\end{equation}
\end{defn}
 
 Given $u\in L^{p}_{loc}(\mathbb{R}^{n})$, $1< p\leq\infty$, and a cube $Q$, we set
\begin{equation}
\Lambda(u;Q)_{L^{p}}:=|Q|^{\frac{-1}{p}}\inf_{C\in\mathbb{R}}\|u-C\|_{L^{p}(Q)}=\inf_{C\in\mathbb{R}}\left(\frac{1}{|Q|}\int_{Q}|u-C|^{p}dx\right)^{\frac{1}{p}},\nonumber
\end{equation}
see Brudnyi \cite{Brud} for the definition of a more general case. Sometimes, $\Lambda(u;Q)_{L^{p}}$ is also called the local oscillation of $u$, for instance, 
see Triebel \cite{Trie}. This quantity is the main object on the theory of local polynomial approximations which provides a unified framework for the description of a large family of spaces of smooth functions. We refer the readers to Brudnyi \cite{Brud1}-\cite{Brud5} for the main ideas and results in local approximation theory.

Given a locally integrable function $u$ on $\mathbb{R}^{n}$, we define its sharp maximal function $u_{1}^{\#}$ by setting
\begin{equation}\label{fsmf}
u_{1}^{\#}(x):=\sup_{r>0}r^{-1}\Lambda(u;Q(x,r))_{L^{1}}.
\end{equation}

In \cite{Calde}, Calder\'on proved that, for $1<p\leq\infty$, a function $u$ is in $W^{1,p}(\mathbb{R}^{n})$, if and only if $u$ and $u_{1}^{\#}$ are both in $L^{p}(\mathbb{R}^{n})$. Moreover, up to constants depending only on $n$ and $p$, we have that
\begin{equation}\label{norm}
\|u\|_{W^{1,p}(\mathbb{R}^{n})}\approx \|u\|_{L^{p}(\mathbb{R}^{n})}+\|u_{1}^{\#}\|_{L^{p}(\mathbb{R}^{n})}.
\end{equation}

This characterization motivates the following definition. Given $1<p\leq\infty$, a function $u\in L^{p}_{loc}(A)$, and a cube $Q$ whose center is in $A$, we let $\Lambda(u;Q)_{L^{p}(A)}$ denote the normalized best approximation of $f$ on $Q$ in $L^{p}$-norm:
\begin{eqnarray}\label{pfsmf}
\Lambda(u;Q)_{L^{p}(A)}&:=&|Q|^{\frac{-1}{p}}\inf_{C\in\mathbb{R}}\|u-C\|_{L^{p}(Q\cap A)}\nonumber\\
                           &=&\inf_{C\in\mathbb{R}}\left(\frac{1}{|Q|}\int_{Q\cap A}|u-C|^{p}dx\right)^{\frac{1}{p}}.
\end{eqnarray}
By $u^{\#}_{1,A}$, we denote the sharp maximal function of $u$ on $A$,
\begin{equation}
u^{\#}_{1,A}(x):=\sup_{r>0}r^{-1}\Lambda(u;Q(x,r))_{L^{1}(A)},\ \ x\in A.\nonumber
\end{equation}
Notice that $u_{1}^{\#}=u^{\#}_{1,\mathbb{R}^{n}}$.

The following trace theorem by Shvartsman from \cite{Shvar}, relates local polynomial approximation to extendability.

\begin{thm}\label{thm1}
Let $A$ be a regular subset of $\mathbb{R}^{n}$. Then a function $u\in L^{p}(A)$, $1<p\leq\infty$, can be extended to a function $Eu\in W^{1,p}(\mathbb{R}^{n})$ if and only if 
\begin{equation}
u_{1,A}^{\#}:=\sup_{r>0}r^{-1}\Lambda(u;Q(\cdot,r))_{L^{1}(A)}\in L^{p}(A).\nonumber
\end{equation}
In addition,
\begin{equation}\label{local norm}
\|u\|_{W^{1,p}(\mathbb{R}^{n})\big|_{A}}\approx\|u\|_{L^{p}(A)}+\|u^{\#}_{1,A}\|_{L^{p}(A)}
\end{equation}
with constants of equivalence depending only on $n,k,p,C_{A}$ and $\delta_{A}$. Here
\begin{equation}
\|u\|_{W^{1,p}(\mathbb{R}^{n})|_{A}}:=\inf\{\|Eu\|_{W^{1,p}(\mathbb{R}^{n})}:Eu\in W^{1,p}(\mathbb{R}^{n}), Eu|_{A}\equiv u\ a.e.\}.\nonumber
\end{equation}
\end{thm}

For a set $A\subset\mathbb{R}^{n}$ of positive Lebesgue measure, we set
\begin{equation}
C^{1,p}(A)=\{u\in L^{p}(A): u_{1,A}^{\#}\in L^{p}(A)\},\ \ \|u\|_{C^{1,p}(A)}=\|u\|_{L^{p}(A)}+\|u_{1,A}^{\#}\|_{L^{p}(A)}.\nonumber
\end{equation}
A result of Haj\l asz, Koskela and Tuominen (Theorem 5 in \cite{HKT}) that partially relies on Theorem \ref{thm1} states the following

\begin{thm}\label{thm2}
Let $\Omega\subset\mathbb{R}^{n}$ be a domain and fix $1<p<\infty$. 
Then the following conditions are equivalent:

$(a)$\ For every $u\in W^{1,p}(\Omega)$ there exists a function $Eu\in W^{1,p}(\mathbb{R}^{n})$ such that $Eu\big|_{\Omega}=f$ a.e.

$(b)$\ $\Omega$ satisfies the measure density condition $(\ref{density})$ and $C^{1,p}(\Omega)=W^{1,p}(\Omega)$ as sets and the norms are equivalent.

$(c)$\ $\Omega\subset\mathbb{R}^{n}$ is a $W^{1,p}$-extension domain.
\end{thm}

In \cite{Shvar}, Shvartsman constructed an extension operator for Theorem \ref{thm1} explicitly as a variant of the Whitney-Jones extension. We describe this procedure in the next section. In particular, based on Theorem \ref{thm2}, for an arbitrary $W^{1,p}$-extension domain $\Omega$ with $1<p<\infty$, there is a Whitney-type extension operator from $W^{1,p}(\Omega)$ to $W^{1,p}(\mathbb{R}^{n})$. For an alternate Whitney-type extension operator see \cite{HKT2}.

\subsection{Whitney type extension}

It will be convenient for us to measure distance via the uniform norm 
\begin{equation}
\|x\|_{\infty}:=\max\{|x_{i}|:i=1,...,n\},\ \ x=(x_{1},...,x_{n})\in\mathbb{R}^{n}.\nonumber
\end{equation}
Thus every Euclidean cube 
\begin{equation}
Q=Q(x,r)=\{y\in\mathbb{R}^{n}:\|y-x\|_{\infty}\leq r\}\nonumber
\end{equation}
is a ball in the $\|\cdot\|_{\infty}$-norm. Given a constant $\lambda>0$, we let $\lambda Q$ denote the cube $Q(x,\lambda r)$. By $Q^{*}$ we denote the cube $Q^{*}:=\frac{9}{8}Q$.

As usual, given subsets $A,B\subset\mathbb{R}^{n}$, we put $\diam A:=\sup\{\|a-a'\|_{\infty}:a,a'\in A\}$ and 
\begin{equation}
\dist (A,B):=\inf\{\|a-b\|_{\infty}:a\in A, b\in B\}.\nonumber
\end{equation}
We also set $\dist(x,A):=\dist(\{x\},A)$ for $x\in\mathbb{R}^{n}$. By $\overline A$ we denote the closure of $A$ in $\mathbb{R}^{n}$ and $\partial A:=\overline A\setminus A$ the boundary of $A$. Finally, $\chi_{A}$ denotes the characteristic function of $A$; we put $\chi_{A}\equiv 0$ if $A=\emptyset$.

The following property for Ahlfors-regular sets is well-known (see, e.g. \cite{Shvar1}).
\begin{lem}\label{closure}
If $A$ is an Ahlfors-regular subset of $\mathbb{R}^{n}$, then $|\partial A|=0$.
\end{lem}

In the remaining part of the paper, we will assume that $S$ is a closed Ahlfors-regular subset of $\mathbb{R}^{n}$. Since now $\mathbb{R}^{n}\setminus S$ is an open set, it admits a Whitney decomposition $W_{S}$ (e.g. see Stein \cite{Stein}). We recall the main properties of $W_{S}$.
\begin{thm}\label{Whitney}
$W_{S}=\{Q_{k}\}$ is a countable family of closed cubes such that

 $(\romannumeral1)$ $\mathbb{R}^{n}\setminus S=\bigcup\{Q:Q\in W_{S}\}$;

 $(\romannumeral2)$ For every cube $Q\in W_{S}$
\begin{equation}
\diam Q\leq\dist (Q,S)\leq 4\diam Q;\nonumber
\end{equation}

 $(\romannumeral3)$ No point of $\mathbb{R}^{n}\setminus S$ is contained in more than $N=N(n)$ distinct cubes from $W_{S}$.
\end{thm}

We also need certain additional properties of Whitney cubes which we present in the next lemma. These properties readily follow from (\romannumeral1)-(\romannumeral3).
\begin{lem}\label{properties}
$(1)$ If $Q,K\in W_{S}$ and $Q^{*}\cap K^{*}\neq\emptyset$, then 
\begin{equation}
\frac{1}{4}\diam Q\leq \diam K\leq 4\diam Q.\nonumber
\end{equation}

$(2)$ For every cube $K\in W_{S}$ there are at most $N=N(n)$ cubes from the family $W_{S}^{*}:=\{Q^{*}:Q\in W_{S}\}$ which intersect $K^{*}$.
\end{lem}

Let $\Phi_{S}:=\{\phi_{Q}:Q\in W_{S}\}$ be a smooth partition of unity subordinated to the Whitney decomposition $W_{S}$, see \cite{Stein}. 

\begin{prop}
$\Phi_{S}$ is a family of functions defined on $\mathbb{R}^{n}$ with the following properties:

(a) $0\leq\phi_{Q}(x)\leq 1$ for every $Q\in W_{S}$;

(b) $supp\phi_{Q}\subset Q^{*} (:=\frac{9}{8}Q), Q\in W_{S}$;

(c) $\sum\{\phi_{Q}(x): Q\in W_{S}\}=1$ for every $x\in\mathbb{R}^{n}\setminus S$;

(d) For every multiindex $\beta$, $|\beta|\leq k$, and every cube $Q\in W_{S}$
\begin{equation}
|D^{\beta}\phi_{Q}(x)|\leq C(\diam Q)^{-|\beta|},\ \ x\in\mathbb{R}^{n},\nonumber
\end{equation} 
where $C$ is a constant depending only on $n$ and $k$.
\end{prop}

Observe that the family of cubes $W_{S}$ constructed in \cite{Stein} satisfies the conditions of Theorem \ref{Whitney} and Lemma \ref{properties} with respect to the Euclidean norm rather than the uniform one. However, a simple modification of this construction provides a family of Whitney cubes which have the analogous properties with respect to the uniform norm. 

Let $K=Q(x_{K},r_{K})\in W_{S}$ and let $a_{K}\in S$ be the point nearest to $x_{K}$ on $S$. Then by the property (\romannumeral2) of Theorem \ref{Whitney},
\begin{equation}
Q(a_{K},r_{K})\subset 10K.\nonumber
\end{equation}

Fix a small $0<\epsilon\leq 1$ and set $K_{\epsilon}:=Q(a_{K},\epsilon r_{K})$. Let $Q=Q(x_{Q},r_{Q})$ be a cube from $W_{S}$ with $\diam Q\leq\delta_{S}$, where $\delta_{S}$ is as in Definition \ref{Ahlfors} for our regular sets. Set 
\begin{equation}
\mathcal{A}_{Q}:=\{K=Q(x_{K},r_{K})\in W_{S}: K_{\epsilon}\cap Q_{\epsilon}\neq\emptyset, r_{K}\leq\epsilon r_{Q}\}.\nonumber
\end{equation} 
(Similar with $K_{\epsilon}$, we set $Q_{\epsilon}:=Q(a_{Q},\epsilon r_{Q})$) we define a ``quasi-cube'' $H_{Q}$ by letting 
\begin{equation}
H_{Q}:=(Q_{\epsilon}\cap S)\setminus\left(\bigcup\{K_{\epsilon}:K\in\mathcal{A}_{Q}\}\right).\nonumber
\end{equation}
If $\diam Q>\delta_{S}$, we put $H_{Q}:=\emptyset$.

The following result is Theorem 2.4 in \cite{Shvar}.
\begin{thm}\label{quasicube}
Let $A$ be a closed regular subset of $\mathbb{R}^{n}$. Then there is a family of ``quasi-cubes" $\mathcal{H}_{\Omega}=\{H_{Q}:Q\in W_{A}\}$ as discussed above with

$(\romannumeral1)$ $H_{Q}\subset(10Q)\cap A, Q\in W_{A}$;

$(\romannumeral2)$ $|Q|\leq\gamma_{1}|H_{Q}|$ whenever $Q\in W_{A}$ with $\diam Q\leq\delta_{A}$;

$(\romannumeral3)$ $\sum_{Q\in W_{A}}\chi_{H_{Q}}\leq\gamma_{2}$.

Here $\gamma_{1}$ and $\gamma_{2}$ are positive constants depending only on $n$ and $C_{A}$.
\end{thm}

Next we present estimates of local polynomial approximations of the extension $Ef$, via corresponding local approximation of a function $f$ defined on a closed regular subset $A\subset\mathbb{R}^{n}$. We start by presenting two lemmas about properties of polynomials on subsets of $\mathbb{R}^{n}$.

Given a measurable subset $A\subset\mathbb{R}^{n}$ and a function $u\in L^{p}(A)$, $1\leq p\leq\infty$, we let $\hat{E_{1}}(u;A)_{L^{p}}$ denote the local best constant approximation in $L^{p}$-norm, see Brudnyi \cite{Brud},
\begin{equation}\label{lbaok}
\hat{E_{1}}(u;A)_{L^{p}}:=\inf_{C\in\mathbb{R}}\|u-C\|_{L^{p}(A)}.
\end{equation}
Thus
\begin{equation}
\Lambda(u;Q)_{L^{p}(A)}=|Q|^{\frac{-1}{p}}\hat{E_{1}}(u;Q\cap A)_{L^{p}}\nonumber
\end{equation}
see (\ref{pfsmf}). We note a simple property of $\Lambda(u;\cdot)_{L^{p}(A)}$ as a function of cubes: for every pair of cubes $Q_{1}\subset Q_{2}$ 
\begin{equation}\label{cube fun}
\Lambda(u;Q_{1})_{L^{p}(A)}\leq\left(\frac{|Q_{2}|}{|Q_{1}|}\right)^{\frac{1}{p}}\Lambda(u;Q_{2})_{L^{p}(A)}.
\end{equation}

Let $A$ be a subset of $\mathbb{R}^{n}$ with $|A|>0$. We put 
\begin{equation}\label{projection}
P_{A}(u):=\bint_{A}u(x)dx=\frac{1}{|A|}\int_{A}u(x)dx.
\end{equation}
Then from a result of Brudnyi in \cite{Brud4}, also see Proposition 3.4 in \cite{Shvar}, we have 
\begin{prop}\label{proposition}
Let $A$ be a subset of a cube $Q$ with $|A|>0$. Then the linear operator $P_{A}:L^{1}(A)\rightarrow\mathbb{R}$ has the property that for every $1\leq p\leq\infty$ and every $u\in L^{p}(A)$
\begin{equation}
\|u-P_{A}(u)\|_{L^{p}(A)}\leq C \hat{E_{1}}(u;A)_{L^{p}}.\nonumber
\end{equation}
Here $C=C(n,\frac{|Q|}{|A|})$.
\end{prop}

According to Lemma \ref{closure}, the boundary of a regular set is of measure zero, so Proposition \ref{proposition} and Theorem \ref{quasicube} immediately imply the following corollary.
\begin{cor}\label{corollary3}
Let $S$ be a closed regular set and  let $Q\in W_{S}$ be a cube with $\diam Q\leq\delta_{S}$. There is a continuous linear operator $P_{H_{Q}}:L^{1}(H_{Q})\rightarrow\mathbb{R}$ such that for every function $u\in L^{p}(S)$, $1\leq p\leq\infty$,
\begin{equation}
\|u-P_{H_{Q}}(u)\|_{L^{p}(H_{Q})}\leq\gamma \hat{E_{1}}(u;H_{Q})_{L^{p}}.\nonumber
\end{equation}
Here $\gamma=\gamma(n,k,\theta_{S})$.
\end{cor}

We put 
\begin{equation}\label{big cube}
P_{H_{Q}}u=0,\ \ {\rm if}\ \ \diam Q>\delta_{S}.
\end{equation}
Now the map $Q\rightarrow P_{H_{Q}}(f)$ is defined on all of the cubes in the family $W_{S}$. This map gives rise to a bounded linear extension operator from $L^{p}(S)$ to $L^{p}(\mathbb{R}^{n})$, which is defined by the formula 
\begin{equation}\label{Extension}
Eu(x):=\left\{\begin{array}{ll}
u(x),&\ \ x\in S,\\

\sum_{Q\in W_{S}}\phi_{Q}(x)P_{H_{Q}}u(x),&\ \  x\in\mathbb{R}^{n}\setminus S.
\end{array}\right.
\end{equation}
 
Given a regular domain $\Omega\subset\mathbb{R}^{n}$, $\overline{\Omega}$ is a closed regular set with $|\overline{\Omega}\setminus\Omega|=0$. Given a function $u\in L^{p}(\Omega)$, the zero extension  of $u$ to the boundary $\overline{\Omega}\setminus\Omega$ (still denoted by $u$) belongs to $L^{p}(S)$, and we define the extension $Eu$ of $u$ to $\mathbb{R}^{n}$ by the formula (\ref{Extension}). When $u\in C^{1,p}(S)$, $Eu$ here is exactly the $Eu$ from Theorem \ref{thm1}.By combining Theorem \ref{thm1} and Theorem \ref{thm2} together, we obtain the following result.

\begin{thm}\label{thm3}
Let $\Omega\subset\mathbb{R}^{n}$ be a $W^{1,p}$-extension domain for some $1<p<\infty$. Then for every $u\in W^{1,p}(\Omega)$ and $Eu$ defined as in (\ref{Extension}) for the zero extension of $u$ to the boundary, we have $Eu\in W^{1,p}(\mathbb{R}^{n})$ and 
\begin{equation}
\|Eu\|_{W^{1,p}(\mathbb{R}^{n})}\leq C\|u\|_{W^{1,p}(\Omega)},\nonumber
\end{equation}
with come positive constant $C$ independent of $u$.
\end{thm}

\section{Proof of Theorem \ref{theorem2}}

The first part of our main theorem (for $1<p<\infty$) will be obtained as a 
consequence of the following extension result.

\begin{thm}\label{theorem1}
Let $\Omega_{1}\subset\mathbb{R}^{n}$ be a $W^{1,p}$-extension domain for some $1<p<\infty$, and $\Omega_{2}\subset\mathbb{R}^{m}$ be a domain. Then for every function $u\in W^{1,p}(\Omega_{1}\times\Omega_{2})$, there exists a function $E_{1}u\in W^{1,p}(\mathbb{R}^{n}\times\Omega_{2})$ such that $E_{1}u\big|_{\Omega_{1}\times\Omega_{2}}\equiv u$ and 
\begin{equation}
\|E_{1}u\|_{W^{1,p}(\mathbb{R}^{n}\times\Omega_{2})}\leq C\|u\|_{W^{1,p}(\Omega_{1}\times\Omega_{2})}\nonumber
\end{equation}
with a positive constant $C$ independent of $u$. 
\end{thm}
\begin{proof}
Theorem 2.3.2 in Ziemer's book \cite{ziemer} tells us that $C^{\infty}(\Omega_{1}\times\Omega_{2})\cap W^{1,p}(\Omega_{1}\times\Omega_{2})$ is dense in $W^{1,p}(\Omega_{1}\times\Omega_{2})$. With a small mollification in the proof of this result, it is easy to see that $C^{1}(\Omega_{1}\times\Omega_{2})\cap L^{\infty}(\Omega_{1}\times\Omega_{2})\cap W^{1,p}(\Omega_{1}\times\Omega_{2})$ is dense in $W^{1,p}(\Omega_{1}\times\Omega_{2})$. We begin by showing that we can extend the functions in $C^{1}(\Omega_{1}\times\Omega_{2})\cap L^{\infty}(\Omega_{1}\times\Omega_{2})\cap W^{1,p}(\Omega_{1}\times\Omega_{2})$.

According to Theorem \ref{thm2}, $\Omega_{1}$ is Ahlfors regular. Let $u\in C^{1}(\Omega_{1}\times\Omega_{2})\cap L^{\infty}(\Omega_{1}\times\Omega_{2})\cap W^{1,p}(\Omega_{1}\times\Omega_{2})$. Then for $y\in\Omega_{2}$, using the extension (\ref{Extension}), we set
\begin{equation}\label{Extension11}
E_{1}u(x,y)=Eu_{y}(x):=\left\{\begin{array}{ll}
u_{y}(x),&\ \ x\in \overline{\Omega_{1}},\\

\sum_{Q\in W_{\Omega_1}}\phi_{Q}(x)P_{H_{Q}}u_{y}(x),&\ \  x\in\mathbb{R}^{n}\setminus\overline{\Omega_{1}}.
\end{array}\right.
\end{equation}
Here $u_{y}$ in (\ref{Extension11}) is the zero extension of $u_{y}$ to the boundary $\partial\Omega_{1}$. In order to show $E_{1}u\in W^{1,p}(\mathbb{R}^{n}\times\Omega_{2})$, we need to show that $E_{1}u\in L^{p}(\mathbb{R}^{n}\times\Omega_{2})$, and for every $\beta$ with $|\beta|=1$, we need to find a function $v_{\beta}\in L^{p}(\mathbb{R}^{n}\times\Omega_{2})$, such that for every $\psi\in C_{0}^{\infty}(\mathbb{R}^{n}\times\Omega_{2})$ we have 
\begin{equation}
\int_{\mathbb{R}^{n}\times\Omega_{2}}E_{1}u(x,y)D^{\beta}\psi(x,y)dxdy=-\int_{\mathbb{R}^{n}\times\Omega_{2}}v_{\beta}(x,y)\psi(x,y)dxdy.\nonumber
\end{equation}

For the convenience of discussion and reading, we divide the rest of proof into three steps.

\textbf{Step 1}: In this step, we show that $E_{1}u\in L^{p}(\mathbb{R}^{n}\times\Omega_{2})$ and that the $L^{p}$-norm of $E_{1}u$ is controlled by the $W^{1,p}$-norm of $u$. By the Fubini theorem, $u_{y}\in W^{1,p}(\Omega_{1})$ for almost every $y\in\Omega_{2}$.  Since $\Omega_{1}$ is a $W^{1,p}$-extension domain, Theorem \ref{thm3} gives that $E_{1}u(x,y)=Eu_{y}(x)\in W^{1,p}(\mathbb{R}^{n})$ and 
\begin{equation}
\|Eu_{y}\|_{L^{p}(\mathbb{R}^{n})}\leq\|Eu_{y}\|_{W^{1,p}(\mathbb{R}^{n})}\leq C\|u_{y}\|_{W^{1,p}(\Omega_{1})},\nonumber
\end{equation}
 for every $y\in\Omega_{2}$ with $u_{y}\in W^{1,p}(\Omega_{1})$. Then by integrating with respect to $y\in\Omega_{2}$, we obtain the desired result.
 
\textbf{Step 2}: In this step, we show that there exist functions $\frac{\partial}{\partial x_{i}}E_{1}u\in L^{p}(\mathbb{R}^{n}\times\Omega_{2})(i=1,...,n)$ such that 
\begin{equation}
\int_{\mathbb{R}^{n}\times\Omega_{2}}\frac{\partial}{\partial x_{i}}E_{1}u(x,y)\psi(x,y)dxdy=-\int_{\mathbb{R}^{n}\times\Omega_{2}}E_{1}u(x,y)\frac{\partial}{\partial x_{i}}\psi(x,y)dxdy\nonumber
\end{equation}
for every $\psi\in C_{c}^{\infty}(\mathbb{R}^{n}\times\Omega_{2})$. For simplicity of notation, we assume that $i=1$.

Fubini's theorem tells us that $u_{y}\in W^{1,p}(\Omega_{1})$ for almost every $y\in \Omega_{2}$. Then by Theorem \ref{thm3}, (\ref{Extension11}) gives an extension $Eu_{y}\in W^{1,p}(\mathbb{R}^{n})$ for every $y\in\Omega_{2}$ with $u_{y}\in W^{1,p}(\Omega_{1})$. Then we set 
\begin{equation}\label{partialx}
\frac{\partial}{\partial x_{1}}E_{1}u(x,y):=\left\{\begin{array}{ll}
\frac{\partial}{\partial x_{1}}Eu_{y}(x),&\ \ {\rm if}\ y\in\Omega_{2}\ {\rm with}\ u_{y}\in W^{1,p}(\Omega_{1}),\\

0,&\ \  {\rm otherwise}.
\end{array}\right.
\end{equation}
Since $Eu_{y}\in W^{1,p}(\mathbb{R}^{n})$ for almost every $y\in \Omega_{2}$, using Fubini's theorem, we obtain
\begin{eqnarray}
\int_{\mathbb{R}^{n}\times\Omega_{2}}\frac{\partial}{\partial x_{1}}E_{1}u(x,y)\psi(x,y)dxdy&=&\int_{\Omega_{2}}\int_{\mathbb{R}^{n}}\frac{\partial}{\partial x_{1}}Eu_{y}(x)\psi(x,y)dxdy\nonumber\\
                                                                                        &=&-\int_{\Omega_{2}}\int_{\mathbb{R}^{n}}Eu_{y}(x)\frac{\partial}{\partial x_{1}}\psi(x,y)dxdy\nonumber\\
																																												&=&-\int_{\mathbb{R}^{n}\times\Omega_{2}}E_{1}u(x,y)\frac{\partial}{\partial x_{1}}\psi(x,y)dxdy,\nonumber
\end{eqnarray}
which means that (\ref{partialx}) gives a first order distributional derivative of $E_{1}u$ with respect to $x_{1}$. Then using the Fubini theorem twice and the fact that the linear operator $E$ from $W^{1,p}(\Omega_{1})$ to $W^{1,p}(\mathbb{R}^{n})$ is bounded, we obtain
\begin{eqnarray}
\int_{\mathbb{R}^{n}\times\Omega_{2}}|\frac{\partial}{\partial x_{1}}E_{1}u(x,y)|^{p}dxdy&=&\int_{\Omega_{2}}\int_{\mathbb{R}^{n}}|\frac{\partial}{\partial x_{1}}Eu_{y}(x)|^{p}dxdy\nonumber\\
                                                                                         &\leq&C\int_{\Omega_{2}}\int_{\Omega_{1}}\left(|u_{y}(x)|^{p}+\big|\frac{\partial}{\partial x_{1}}u_{y}(x)\big|^{p}\right)dxdy\nonumber\\
																																												&\leq&C\int_{\Omega_{1}\times\Omega_{2}}\left(|u(x,y)|^{p}+\big|\frac{\partial}{\partial x_{1}}u(x,y)\big|^{p}\right)dxdy,\nonumber
\end{eqnarray}
we have obtained the desired norm estimate.

\textbf{Step 3}: In this step, we show that there exist functions $\frac{\partial}{\partial y_{j}}E_{1}u\in L^{p}(\mathbb{R}^{n}\times\Omega_{2}) (j=1,...,m)$ such that 
\begin{equation}
\int_{\mathbb{R}^{n}\times\Omega_{2}}\frac{\partial}{\partial y_{j}}E_{1}u(x,y)\psi(x,y)dxdy=-\int_{\mathbb{R}^{n}\times\Omega_{2}}E_{1}u(x,y)\frac{\partial}{\partial y_{j}}\psi(x,y)dxdy\nonumber
\end{equation}
for every $\psi\in C_{c}^{\infty}(\mathbb{R}^{n}\times\Omega_{2})$. For simplicity of notation, we assume that $j=1$.

Consider the projection 
\begin{equation} 
P_{1}:\Omega_{2}\rightarrow\mathbb{R}^{m-1},\nonumber
\end{equation}
which is defined by setting
\begin{equation}
 P_{1}(y)=(y_{2},y_{3},...,y_{m})=:\check{y}_{1}\ {\rm for}\ y=(y_{1},...,y_{m})\in\Omega_{2}.\nonumber
\end{equation}
Set $S_{1}^{\check{y}_{1}}:=P_{1}^{-1}(\check{y}_{1})\subset\Omega_{2}$, the preimage of $\check{y}_{1}\in P_{1}(\Omega_{2})$. Then $S_{1}^{\check{y}_{1}}$ is the union of at most countably many pairwise disjoint segments.  

Fix $x\in\mathbb{R}^{n}\setminus\overline{\Omega_{1}}$ and $\check{y}_{1}\in P_{1}(\Omega_{2})$. To begin, we assume that $S_{1}^{\check{y}_{1}}$ is a single segment. Now for $y_{1}^{1},\check{y}_{1}, y_{1}^{2},\check{y}_{1}\in S_{1}^{\check{y}_{1}}$, according to (\ref{Extension11}), we have
\begin{eqnarray}\label{eq}
E_{1}u(x,y_{1}^{1},\check{y}_{1})&-&E_{1}u(x,y_{1}^{2},\check{y}_{1})\\
                                &=&\sum_{Q\in W_{\overline{\Omega_{1}}}}\phi_{Q}(x)\left((P_{H_{Q}}u(x, y_{1}^{1},\check{y}_{1})-(P_{H_{Q}}u(x, y_{1}^{2},\check{y}_{1})\right).\nonumber
\end{eqnarray}
 By the definition (\ref{projection}) of $P_{H_{Q}}u$ and the facts that $u$ is $C^{1}$ and $H_{Q}\times S_{1}^{\check{y}_{1}}\subset\Omega_{1}\times\Omega_{2}$, we have
\begin{eqnarray}\label{pq}
(P_{H_{Q}}u)(x, y_{1}^{1},\check{y}_{1})&-&(P_{H_{Q}}u)(x, y_{1}^{2},\check{y}_{1})\\
                                      &=&\bint_{H_{Q}}\left(u(w,y_{1}^{1},\check{y}_{1})-u(w, y_{1}^{2},\check{y}_{1})\right)dw \nonumber\\
																			&=&\bint_{H_{Q}}\left(\int_{y_{1}^{2}}^{y_{1}^{1}}\frac{\partial u(w, s,\check{y}_{1})}{\partial y_{1}}ds\right)dw.\nonumber
\end{eqnarray}
Combining (\ref{eq}) and (\ref{pq}), we obtain 
\begin{eqnarray}\label{variation}
E_{1}u(x, y_{1}^{1},\check{y}_{1})&-&E_{1}u(x, y_{1}^{2},\check{y}_{1})\\
                                &=&\sum_{Q\in W_{\overline{\Omega_{1}}}}\phi_{Q}(x)\bint_{H_{Q}}\int_{y_{1}^{2}}^{y_{1}^{1}}\frac{\partial u(w,s,\check{y}_{1})}{\partial y_{1}}dsdw;\nonumber
\end{eqnarray}
notice that $x$ is contained in the support of only finite many $\Phi_{Q}$, hence $E_{1}u(x,s,\check{y}_{1})$ is absolutely continuous as a function of $s$ on $S_{1}^{\check{y}_{1}}$. By repeating this for each component of $S_{1}^{\check{y}_{1}}$, we conclude that $E_{1}u(x,s,\check{y}_{1})$ is absolutely continuous as a function of $s$ on every component of $S_{1}^{\check{y}_{1}}$. From (\ref{variation}) and the Lebesgue differentiation theorem, we deduce that 
\begin{eqnarray}\label{commutative}
\frac{\partial E_{1}u(x,s,\check{y}_{1})}{\partial y_{1}}&:=&\lim_{s'\rightarrow s}\frac{E_{1}u(x,s',\check{y}_{1})-E_{1}u(x,s,\check{y}_{1})}{s'-s}\\\
                                                                                    &=&\sum_{Q\in W_{\overline{\Omega_{1}}}}\phi_{Q}(x)\bint_{H_{Q}}\frac{\partial u(w,s,\check{y}_{1})}{\partial y_{1}}dw=E_{1}\frac{\partial u(x, s,\check{y}_{1})}{\partial y_{1}},\nonumber
\end{eqnarray}
exists for $\mathcal{H}^{1}-a.e.$ $s$ with $(s,\check{y}_{1})\in S_{1}^{\check{y}_{1}}$. Fix $\psi\in C^{\infty}_{c}(\mathbb{R}^{n}\times\Omega_{2})$. Since $E_{1}u(x,s,\check{y}_{1})$ is absolutely continuous as a function of $s$ on each segment of $S_{1}^{\check{y}_{1}}$, we conclude that 
\begin{eqnarray}
\int_{S_{1}^{\check{y}_{1}}}E_{1}u(x,s,\check{y}_{1})\frac{\partial\psi(x,s,\check{y}_{1})}{\partial y_{1}}ds
																																								 &=&-\int_{S_{1}^{\check{y}_{1}}}\frac{\partial E_{1}u(x,s,\check{y}_{1})}{\partial y_{1}}\psi(x,s,\check{y}_{1})ds.\nonumber
\end{eqnarray}

In order to complete the definition of $\frac{\partial E_{1}u}{\partial y_{1}}$, we define $\frac{\partial E_{1}u}{\partial y_{1}}=\frac{\partial u}{\partial y_{1}}$ when $(x,y)\in\Omega_{1}\times\Omega_{2}$ and $\frac{\partial E_{1}u}{\partial y_{1}}=0$ when $(x,y)\in\partial\Omega_{1}\times\Omega_{2}$. Then let us show that $\frac{\partial E_{1}u}{\partial y_{1}}$ is a first order distributional derivative of $E_{1}u$ with respect to $y_{1}$-coordinate. By the Fubini theorem, (\ref{commutative}) and the fact that $|\partial\Omega_{1}|=0$, we have  
\begin{eqnarray} 
\int_{\mathbb{R}^{n}\times\Omega_{2}}E_{1}u(x,y)\frac{\partial\psi(x,y)}{\partial y_{1}}dxdy&=&\int_{\mathbb{R}^{n}}\int_{P_{1}(\Omega_{2})}\int_{S_{1}^{\check{y}_{1}}}E_{1}u(x,y)\frac{\partial\psi(x,y)}{\partial y_{1}}dy_{1}d\check{y}_{1}dx\nonumber\\
																																												    &=&-\int_{\mathbb{R}^{n}}\int_{P_{1}(\Omega_{2})}\int_{S_{1}^{\check{y}_{1}}}\frac{\partial E_{1}u(x,y)}{\partial y_{1}}\psi(x,y)dy_{1}d\check{y}_{1}dx\nonumber\\
																																												    &=&-\int_{\mathbb{R}^{n}\times\Omega_{2}}\frac{\partial E_{1}u(x,y)}{\partial y_{1}}\psi(x,y)dxdy.\nonumber
\end{eqnarray}

Now we show that $\frac{\partial E_{1}u}{\partial y_{1}}\in L^{p}(\mathbb{R}^{n}\times\Omega_{2})$ and that its norm is controlled by the Sobolev norm of $u$. Since $|\partial\Omega_{1}|=0$, we have
\begin{eqnarray}
\int_{\mathbb{R}^{n}\times\Omega_{2}}\big|\frac{\partial E_{1}u(x,y)}{\partial y_{1}}\big|^{p}dxdy
                                                                                              &=&\int_{\Omega_{1}\times\Omega_{2}}\big|\frac{\partial u(x,y)}{\partial y_{1}}\big|^{p}dxdy\nonumber\\
																																															& &+\int_{(\mathbb{R}^{n}\setminus\overline{\Omega_{1}})\times\Omega_{2}}\big|E_{1}\frac{\partial u(x,y)}{\partial y_{1}}\big|^{p}dxdy.\nonumber
\end{eqnarray}

As we know, for almost every $y\in\Omega_{2}$, $\frac{\partial u}{\partial y_{1}}\big|_{y}\in L^{p}(\Omega_{1})$. Using the fact that $E: L^{p}(\Omega_{1})\rightarrow L^{p}(\mathbb{R}^{n})$ is a bounded linear operator, we obtain
\begin{equation}
\int_{\mathbb{R}^{n}}\big|E_{1}\frac{\partial u(x,y)}{\partial y_{1}}\big|^{p}dx\leq C\int_{\Omega_{1}}\big|\frac{\partial u(x,y)}{\partial y_{1}}\big|^{p}dx,\nonumber
\end{equation}
for almost every $y\in\Omega_{2}$. Then we do the integration with respect to $y\in \Omega_{2}$ on the two sides of the inequality above, we obtain the desired inequality
\begin{equation}
\int_{\mathbb{R}^{n}\times\Omega_{2}}\big|\frac{\partial E_{1}u(x,y)}{\partial y_{1}}\big|^{p}dxdy\leq C\int_{\Omega_{1}\times\Omega_{2}}\big|\frac{\partial u(x,y)}{\partial y_{1}}\big|^{p}dxdy.\nonumber
\end{equation}

In conclusion, we have showed that the linear extension operator $E_{1}$ is bounded from $C^{1}(\Omega_{1}\times\Omega_{2})\cap L^{\infty}(\Omega_{1}\times\Omega_{2})\cap W^{1,p}(\Omega_{1}\times\Omega_{2})$ to $W^{1,p}(\mathbb{R}^{n}\times\Omega_{2})$ for our fixed $1<p<\infty$. Since $C^{1}(\Omega_{1}\times\Omega_{2})\cap L^{\infty}(\Omega_{1}\times\Omega_{2})\cap W^{1,p}(\Omega_{1}\times\Omega_{2})$ is dense in $W^{1,p}(\Omega_{1}\times\Omega_{2})$, $E_{1}$ extends to a bounded linear extension operator from $W^{1,p}(\Omega_{1}\times\Omega_{2})$ to $W^{1,p}(\mathbb{R}^{n}\times\Omega_{2})$.
\end{proof}

\begin{proof}[Proof of Theorem~\ref{theorem2}]
Regarding the first part of the claim,
by Theorem~\ref{theorem1} we have a bounded extension operator 
$E_1:W^{1,p}(\Omega_1\times \Omega_2)\to W^{1,p}(\mathbb{R}^{n}\times\Omega_{2}),$
and it thus suffices to extend functions in  
$W^{1,p}(\mathbb{R}^{n}\times\Omega_{2})$ to 
$W^{1,p}(\mathbb{R}^{n}\times \mathbb{R}^m).$ Given $u\in W^{1,p}(\mathbb{R}^{n}\times\Omega_{2}),$ define $\hat u(x,y)=u(y,x).$ Then $\hat u \in 
W^{1,p}(\Omega_2\times \mathbb{R}^n)$ and the desired extension is obtained via
Theorem~\ref{theorem1} as $\Omega_2\subset \mathbb{R}^m$ is a $W^{1,p}$-extension domain.

Towards the second part, by symmetry, it suffices to prove that 
$\Omega_1\subset \mathbb{R}^n$ must be a $W^{1,p}$-extension domain whenever 
$\Omega_1\times \Omega_2$ is such a domain. 

Suppose first that $\Omega_2$ has finite measure. Given $u\in W^{1,p}(\Omega),$
define $v(x,y)=u(x).$ Then $v\in W^{1,p}(\Omega_1\times \Omega_2).$ Let
$Ev\in  W^{1,p}(\mathbb{R}^n\times \mathbb{R}^m)$ be an extension of $v.$ Then
$Ev\in  W^{1,p}(\mathbb{R}^n\times \{y\})$ for almost every $y\in \Omega_2.$
This follows via the Fubini theorem from the ACL-characterization of $W^{1,p}$
given in our introduction.
Since $v(x,y)=u(x),$ we conclude that $u$ must be the restriction of some
function $w\in W^{1,p}(\mathbb{R}^n).$ This allows us to infer from 
Theorem~\ref{thm2} that $\Omega_1$ must be a $W^{1,p}$-extension domain.

In case $\Omega_2$ has infinite measure, we fix a ball $B\subset \Omega_2$ and
pick a smooth function $\psi$ with compact support so that $\psi$ is identically $1$ on $B.$ We still define $v$ as above and set $w=\psi v.$ Then $w\in W^{1,p}(\Omega_1\times \Omega_2)$ and we may repeat the above argument as $w(x,y)=u(x)$ for almost every $y\in B\subset \Omega_2.$
\end{proof}

\noindent Pekka Koskela\\
  Zheng Zhu

\medskip

\noindent
Department of Mathematics and Statistics, University of Jyv\"askyl\"a, FI-40014, Finland.

\noindent{\it E-mail address}:  \texttt{pekka.j.koskela@jyu.fi}\\
                                 \texttt{zheng.z.zhu@jyu.fi}


\begin{thebibliography}{99}

\vspace{-0.3cm}
\bibitem{Brud1}
Yu. A. Brudnyi, 
A multidimensional analog of a certain theorem of Whitney,
Mat. Sb. (N.S.)82(124), 169-191 (1970); English transl.: Math. USSR-Izv. 4, 568-586 (1970).

\vspace{-0.3cm}
\bibitem{Brud2}
Yu. A. Brudnyi,
Approximations of functions of $n$ variables by quasi-polynomials, 
Izv. Akad. Nauk SSSR Ser. Mat. 34,564-583 (1970); English transl.: Math. USSR-Izv, 4. 568-596 (1970). 

\vspace{-0.3cm}
\bibitem{Brud}
Yu. A. Brudnyi, 
Spaces that are definable by means of local approximations,
Trudy Moskov. Math. Obshch. 24, 69-132 (1971); English transl.: Trans. Moscow Math Soc. 24, 73-139 (1974).

\vspace{-0.3cm}
\bibitem{Brud3}
Yu. A. Brudnyi, 
Piecewise polynomial approximation, embedding theorem and rational approximation,
Lecture Notes in Mathematics Vol. 556 (Springer-Verlag, Berlin, 1976), pp. 73-98.

\vspace{-0.3cm}
\bibitem{Brud4}
Yu. A. Brudnyi, 
Investigation in the Theory of Local Approximations, Doctoral Dissertation, Leningrad University (1977) (in Russian).

\vspace{-0.3cm}
\bibitem{Brud5}
Yu. A. Brudnyi,
Adaptive approximation of functions with singularities,
Trudy Moskov. Mat. Obshch. 55, 149-242 (1994); English transl.: Trans. Moscow Math. Soc. 55, 123-186 (1995).

\vspace{-0.3cm}
\bibitem{BruGa}
Yu. A. Brudnyi and M. I. Ganzburg,
On an Extremal Problem for Polynomials in $n$ Variables,
Izv. Akad. Nauk SSSR Ser. Mat. 37, 344-355 (1973); English transl.: Math. USSR Izv. 7, 345-356 (1973).

\vspace{-0.3cm}
\bibitem{Calde}
A. P. Calder\'on,
Estimates for singular integral operators in terms of maximal functions,
Studia Math. 44, 563-582 (1972).

\vspace{-0.3cm}
\bibitem{HKT}
P. Haj\l asz, P. Koskela and H. Tuominen,
Sobolev embeddings, extensions and measure density condition,
J. Funct. Anal. 254 (2008), no.5, 1217-1234.

\vspace{-0.3cm}
\bibitem{HKT2}
P. Haj\l asz, P. Koskela and H. Tuominen,
Measure density and extendability of Sobolev functions,
Rev. Mat. Iberoam. 24 (2008), no. 2, 645-669.

\vspace{-0.3cm}
\bibitem{Pekka}
S. Hencl and P. Koskela,
Lectures on Mappings of finite distortion, 
Lecture Notes in Mathematics. 2096, Springer, Cham, 2014.

\vspace{-0.3cm}
\bibitem{KRZ}
P. Koskela, T. Rajala and Y. Zhang,
A geometric characterization of planar Sobolev extension domains, 
preprint.


\vspace{-0.3cm}
\bibitem{Pekka98}
P. Koskela,
Extensions and imbeddings,
J. Funct. Anal. 159, 369-384 (1998)

\vspace{-0.3cm}
\bibitem{Shvar}
P. Shvartsman,
Local approximations and intrinsic characterizations of spaces of smooth functions on regular subset of $\mathbb{R}^{n}$,
Math. Nachr. 279 (11) (2006) 1212-1241.

\vspace{-0.3cm}
\bibitem{Shvar1}
P. Shvartsman, 
On extension of Sobolev functions defined on regular subsets of metric measure spaces, 
J. Approx. Theory. 144 (2007) 139–161.

\vspace{-0.3cm}
\bibitem{Shvar2}
P. Shvartsman, 
On Sobolev extension domains in $\mathbb{R}^{n}$,
J. Funct. Anal. 258 (2010), no. 7, 2205-2245.


\vspace{-0.3cm}
\bibitem{Stein}
E. M. Stein,
Singular Integrals and Differentiability Properties of Functions (Princeton Univ. Press, Princeton, NJ, 1970).

\vspace{-0.3cm}
\bibitem{Trie}
H. Triebel,
Theory of Function Spaces. \uppercase\expandafter{\romannumeral2},
Monographs in Mathematics Vol. 84 (Birkh\"auser, Basel, 1992).

\vspace{-0.3cm}
\bibitem{Whitney}
H. Whitney,
Functions differentiable on the boundaries of regions,
Ann of Math. 35(3) (1934) 482-485.

\vspace{-0.3cm}
\bibitem{ziemer}
W. P. Ziemer,
Weakly Differentiable Functions,
Graduate Texts in Mathematics, vol. 120. New York: Springer-verlag. 

\vspace{-0.3cm}
\bibitem{zobin}
N. Zobin,
Whitney's problem on extendability of functions and an intrinsic metric,
Adv. Math. 133 (1998) 96-132.

\end{thebibliography}
\end{document}